\documentclass{amsart}
\usepackage{amssymb, amsmath, amsthm, graphics, comment, xspace, enumerate}
\usepackage{hyperref}
\usepackage{color}
\newtheorem{thm}{Theorem}[section]

\newtheorem{lem}[thm]{Lemma}

\theoremstyle{definition}
\newtheorem{defn}[thm]{Definition}
\newtheorem{example}[thm]{Example}
\theoremstyle{remark}
\newtheorem{rem}[thm]{Remark}
\numberwithin{equation}{section}

\begin{document}
\title[Abstract nonautonomous difference inclusions...]{Abstract nonautonomous difference inclusions in locally convex spaces}

\author{Marko Kosti\' c}
\address{Faculty of Technical Sciences,
University of Novi Sad,
Trg D. Obradovi\' ca 6, 21125 Novi Sad, Serbia}
\email{marco.s@verat.net}

{\renewcommand{\thefootnote}{} \footnote{
\\ \text{  }  \ \    {\it Key words and phrases.} Abstract  nonautonomous difference inclusions, almost periodic sequences, Weyl almost periodic  sequences, Besicovitch almost periodic  sequences, locally convex spaces.}}

\begin{abstract}
In this paper, we consider abstract nonautonomous difference inclusions in locally convex spaces with integer order differences. We particularly analyze the existence and uniqueness of almost periodic type solutions to abstract nonautonomous difference inclusions. Our results seem to be completely new even in the Banach space setting.
\end{abstract}
\maketitle

\section{Introduction and preliminaries}

Suppose that $(Y, \| \cdot \|)$ is a complex Banach space.
Then we say that
a $Y$-valued sequence $(x(k))_{k\in {\mathbb Z}}$ 
is (Bohr)
almost periodic if, for every $\epsilon> 0,$ there exists $l>0$ such that for each $t \in {\mathbb Z}$ there exists ${\bf \tau} \in {\mathbb Z}\cap [t,t+l]$ such that  
\begin{align*}
\bigl\|x(k+{\bf \tau})-x(k)\bigr\| \leq \epsilon,\quad k\in {\mathbb Z}.
\end{align*}
Any almost periodic $Y$-valued sequence is bounded and its range is relatively compact in $Y.$ 
We know that a sequence $(x(k))_{k\in {\mathbb Z}}$  in $Y$ is almost periodic if and only if there exists an almost periodic function $F : {\mathbb R} \rightarrow Y$
such that $x(k)=F(k)$ for all $k\in {\mathbb Z};$ for more details about almost periodic functions, almost periodic sequences and their applications, we refer to research monographs \cite{diagana,gaston,nova-mono,nova-selected,188,30}.

The abstract  difference equations with integer order differences
have been considered in the research monograph \cite{gil} by M. I. Gil, where all considered operator coefficients are bounded linear operators between normed spaces. In our recent research monograph \cite{funkcionalne}, we have investigated the existence and uniqueness of almost periodic type solutions for various classes of abstract difference inclusions in Banach spaces with multivalued linear operators as operator coefficients. 

On the other hand, the almost periodic sequences with values in locally convex spaces and the abstract difference equations in locally convex spaces 
have not received so much attention of the authors by now; cf. \cite{gefter,gefter1,iwasaki,mosle,novac} for some results obtained in this direction. 
In \cite{nizovi}, we have recently analyzed various classes of generalized $\rho$-almost periodic type sequences with values in locally convex spaces and presented some applications to abstract Volterra difference equations.

The main aim of this paper is to continue the research study raised in \cite{nizovi} and the research studies \cite{zh1}-\cite{zh2} by S. Zhang. In \cite{zh1}, the author has investigated
the existence and uniqueness of almost periodic solutions to nonautonomous difference equation
\begin{align}\label{szhang}
x(k+1)=A(k)x(k)+b(k),\quad k\in {\mathbb Z},
\end{align}
where $s\in {\mathbb N},$ $(A(k))_{k\in {\mathbb Z}}$ is an almost periodic real matrix of format $s\times s$ and $(b(k))_{k\in {\mathbb Z}}$ is an almost periodic sequence in ${\mathbb R}^{s}.$ It has been proved that the system \eqref{szhang} has a unique almost periodic solution provided that
$$
\overline{A}=\limsup_{k\rightarrow +\infty}\frac{1}{k}\sum_{j=1}^{k}\| A(j)\|<1,
$$
where $\|A(j)\|$ denotes the sum of all absolute values of the elements of $A(j);$ moreover, if $(A(k))_{k\in {\mathbb Z}}$ and $(b(k))_{k\in {\mathbb Z}}$ are of period $\omega \in {\mathbb N},$ then the solution $(x(k))_{k\in {\mathbb Z}}$ is also of period $\omega$. 
In this paper, we first consider the following operator extension of nonautonomous difference equation \eqref{szhang}:
\begin{align}\label{zajeb-equ00}
x(k+1)=A(k)x(k)+f(k),\quad k\in {\mathbb Z},
\end{align}
where $A(k)$ is a linear continuous operator acting on a Hausdorff sequentially complete locally convex space $Y$ ($k\in {\mathbb Z}$) and $f: {\mathbb Z} \rightarrow Y;$ after that, we apply the established result to the abstract nonautonomous difference inclusions with multivalued linear operators and the abstract nonautonomous (degenerate) difference equations with linear operators. Concerning the abstract nonautonomous difference equations in Banach spaces and the 
geometric theory of discrete nonautonomous dynamical systems, we may also refer to \cite{aul,megan2,bar,megan1,migda,kristijan,russ} and references quoted therein.

The organization and main ideas of this research article can be briefly described as follows. We first introduce the notation and some preliminaries; after that, we recall the basic definitions and results about Bohr almost periodic sequences, Weyl almost periodic sequences and Besicovitch almost periodic sequences (cf. Subsection \ref{s1n}). Our main results concerning the abstract nonautonomous first-order difference equations are given in Section \ref{s2n}. In Theorem \ref{zajeb-e}, we analyze 
the existence and uniqueness of almost periodic type solutions to \eqref{zajeb-equ00}. After that, we analyze the existence and uniqueness of almost periodic type solutions to the following abstract nonautonomous difference inclusion
\begin{align*}
Cx(k+1)\in {\mathcal A}(k)x(k)+Cf(k),\quad k\in {\mathbb Z},
\end{align*} 
where $C$ is a linear continuous operator on $Y,$ ${\mathcal A}(k)$ is an MLO in $Y$ and $[{\mathcal A}(k)]^{-1}C\in L(Y)$ for all $k\in {\mathbb Z};$ cf. Theorem \ref{zajeb-mlo}. 
The abstract difference equations
\begin{align*}
CB(k+1)u(k+1)=A(k)u(k)+Cf(k),\quad k\in {\mathbb Z}
\end{align*}
and 
\begin{align*}
B(k+1)Cu(k+1)=A(k)u(k)+Cg(k),\quad k\in {\mathbb Z},
\end{align*}
where $C$ is a linear continuous operator on $Y,$ $A(k)$ and $B(k)$ are linear operators on $Y$ ($k\in {\mathbb Z}$), $f: {\mathbb Z} \rightarrow Y$ and $g: {\mathbb Z} \rightarrow Y$ are given sequences, are analyzed in Theorem \ref{zajeb-mlo1} and Theorem \ref{zajeb-mlo2}, respectively. 

The abstract nonautonomous higher-order difference equations are investigated in Section \ref{s2nn}. In particular, we consider the following abstract nonautonomous second-order difference equation:
\begin{align*}
CA_{2}(k+2) u(k+2)+CA_{1}(k+1)u(k+1)+A_{0}u(k)=Cf(k),\quad k\in {\mathbb Z},
\end{align*}
where $f: {\mathbb Z} \rightarrow Y$, $C$ is a linear continuous operator on $Y$ and $A_{2}(k),\ A_{1}(k),\ A_{0}(k)$ are linear operators on $Y$ for all $k\in {\mathbb Z}.$ 

We provide several illustrative examples, observations and useful remarks about considered themes.
 As mentioned in the abstract, our results seem to be completely new even in the setting of Banach spaces; here, we also present some new applications in $E_{l}$-type spaces and their projective limits (cf. \cite{x263} and \cite{FKP} for more details about the generation of various classes of $(a,k)$-regularized $C$-resolvent families in $E_{l}$-type spaces, an important class of Fr\' echet function spaces). Before proceeding any further, we would like to emphasize that this is probably the first research article in the existing literature which concerns the use of $C$-regularized resolvents of (multivalued) linear operators in the theory of abstract nonautonomous difference equations.
\vspace{0.1cm}

\noindent {\bf Notation and preliminaries.} If  $Y$ is a Hausdorff locally convex space\index{sequentially complete locally convex space!Hausdorff} over the field of complex numbers, then
the abbreviation $\circledast_{Y}$ stands for the fundamental system of seminorms\index{system of seminorms} which defines the topology of $Y$; ${\rm I}$ denotes the identity operator on $Y$. If $Y$ is sequentially complete, then we simply write that $Y$ is an SCLCS. If $X$ is also a Hausdorff locally convex space over the field ${\mathbb C}$, 
then $L(X,Y)$ denotes the space consisting of all continuous linear mappings\index{continuous linear mapping} from $X$ into $Y$; $L(X)\equiv L(X,X)$. If $Y_{1},\ ..., \ Y_{k}$ are locally convex spaces, then $Y_{1}\times ...\times Y_{k}$ is a locally convex space and the fundamental system of seminorms which defines the topology on $Y_{1}\times ...\times Y_{k}$ is given by $\kappa(y_{1},...,y_{k}):=\kappa_{1}(y_{1})+...+\kappa_{k}(y_{k})$, $y_{1}\in Y_{1},...,$ $y_{k}\in Y_{k}$, where $\kappa_{j}(\cdot)$ runs through the set $\circledast_{j}$ for $1\leq j\leq k.$ 
For more details about the multivalued linear operators (MLOs) in locally convex spaces, we refer the reader to \cite{FKP}; we will use the same notion and notation as in this monograph.
By a trigonometric polynomial $P : {\mathbb R} \rightarrow Y$ we mean any linear combination of functions like
$
e^{i \lambda t}y,
$
where $\lambda$ is a real number and $y\in  Y$. By
$\Gamma (\cdot)$ we denote the Euler Gamma function; finally, if
$\gamma>0,$ then we define $g_{\gamma}(t):=t^{\gamma-1}/\Gamma(\gamma),$ $t>0.$

\subsection{Generalized almost periodic type sequences in locally convex spaces}\label{s1n}

In the sequel, we will always assume that $Y$ is an SCLCS.
If $F : {\mathbb Z} \rightarrow Y$ is a given sequence, $\omega \in {\mathbb N}$ and $c\in {\mathbb C}\setminus \{0\},$  then we say that $F(\cdot)$ is $(\omega,c)$-periodic if $F(k+\omega)=cF(k)$ for all $k\in {\mathbb Z};$ cf. \cite{l}, \cite{fila} and the references cited therein for more details about $(\omega,c)$-(almost) periodic functions and their applications given in the Banach space setting.
We need the following notion, as well (cf. \cite{nizovi} for more details):

\begin{defn}\label{trbusacc}
Suppose that $p>0$ and $F : {\mathbb Z}\rightarrow Y$ is a given sequence. Then we say that: 
\begin{itemize}
\item[(i)] $F(\cdot)$ is (Bohr) almost periodic if, for every $\epsilon>0$ and $\kappa \in \circledast_{Y}$, there exists a finite real number
$l>0$ such that for each $t\in {\mathbb Z}$ there exists $\tau \in {\mathbb Z} \cap [t,t+l]$ such that
\begin{align*}
\kappa\bigl(F(k+{\bf \tau})-F(k)\bigr) \leq \epsilon,\quad k\in {\mathbb Z};
\end{align*}
\item[(ii)] $F(\cdot)$ belongs to the class $B-e-W_{ap}^{p}({\mathbb Z}  : Y)$ if, for every $\epsilon>0$ and $\kappa \in \circledast_{Y}$, there exist a trigonometric polynomial $P:  {\mathbb Z} \rightarrow Y$ and an integer $l\in {\mathbb N}$ such that 
\begin{align*}
l^{-1}\sum_{j=s}^{s+l} \Bigl[\kappa\bigl( F(j)-P(j)\bigr)\Bigr]^{p} \leq \epsilon ,\quad s\in {\mathbb Z};
\end{align*}
\item[(iii)] $F(\cdot)$ is Besicovitch-$p$-almost periodic if, for every $\epsilon>0$ and $\kappa \in \circledast_{Y}$, there exists a trigonometric polynomial $P:  {\mathbb Z} \rightarrow Y$ such that
\begin{align*}
\limsup_{l\rightarrow +\infty}\Biggl\{ l^{-1}\cdot \sum_{j=-l}^{l}\Bigl[\kappa\bigl( F(j)-P(j)\bigr)\Bigr]^{p}\Biggr\}<\epsilon.
\end{align*}
\end{itemize}
\end{defn}

It can be simply proved that, in parts (ii) and (iii), we can use a Bohr almost periodic function $H :  {\mathbb Z} \rightarrow Y$ in place of the trigonometric polynomial $P:  {\mathbb Z}\rightarrow Y$. Using the discrete H\"older inequality and a simple argumentation (cf. also the proof of \cite[Proposition 2.3]{abas-vlade}), we can simply show that a bounded sequence $F : {\mathbb Z} \rightarrow Y$ belongs to the class $B-e-W_{ap}^{p}({\mathbb Z} : Y)$ [is Besicovitch-$p$-almost periodic] if and only if $F(\cdot)$ belongs to the class $B-e-W_{ap}^{1}({\mathbb Z} : Y)$ [is Besicovitch-$1$-almost periodic]; cf. also the formulation of Theorem \ref{zajeb-e}(iv)-(v) below, where we have assumed that $p=1.$

We need the following results, which can be deduced in the same manner as for Bohr almost periodic functions (cf. \cite{note} for more details):

\begin{lem}\label{bochner1233210}
\begin{itemize}
\item[(i)]
Suppose that $\alpha \in {\mathbb C}$ and $\beta \in {\mathbb C}$. If 
the  sequences $F : {\mathbb Z}  \rightarrow Y$ and $G : {\mathbb Z} \rightarrow Y$ are Bohr almost periodic, then the sequence $(\alpha F+\beta G)(\cdot)$ is likewise Bohr almost periodic.
\item[(ii)]
Suppose that $k\in {\mathbb N}$, $Y_{i}$ is a locally convex space for $1\leq i\leq k$ and 
the sequence $F_{i} : {\mathbb Z} \rightarrow Y_{i}$ is Bohr almost periodic for $1\leq i\leq k$. Then the sequence $(F_{1},...,F_{k}): {\mathbb Z}\rightarrow Y_{1}\times ...\times Y_{k}$ is Bohr almost periodic.
\end{itemize}
\end{lem}

For the sake of simplicity and better exposition, we will not consider the abstract semilinear nonautonomous difference inclusions henceforth.

\section{Abstract nonautonomous first-order difference inclusions}\label{s2n}

Let us consider the abstract nonautonomous difference equation
\begin{align}\label{zajeb-equ}
x(k+1)=A(k)x(k)+f(k),\quad k\in {\mathbb Z},
\end{align}
where $A(k)\in L(Y),$  $k\in {\mathbb Z}.$ Our first structural result reads as follows:

\begin{thm}\label{zajeb-e}
Suppose that for each $k\in {\mathbb Z}$ and $\kappa \in \circledast_{Y}$ there exists a finite real number $c^{\kappa}(k)>0$ such that $\kappa(A(k)x)\leq c^{\kappa}(k)\kappa(x)$ for all $x\in Y$, and 
\begin{align}\label{rac}
\sum_{v=1}^{+\infty}\Bigl[ c^{\kappa}(k-v)\cdot ... \cdot c^{\kappa}(k-1)\Bigr] <+\infty , \quad k\in {\mathbb Z}.
\end{align}
Suppose, further, that for each $\kappa \in \circledast_{Y}$ we have $\sup_{k\in {\mathbb Z}}c^{\kappa}(k)<+\infty .$ 
Then we have the following:
\begin{itemize}
\item[(i)] If the sequence $f : {\mathbb Z} \rightarrow Y$ is bounded, then the sequence
\begin{align}\label{mileva}
x(k):=f(k-1)+\sum_{j=-\infty}^{k-2}\Biggl[ A(k-1)\cdot ...\cdot A(j+1)\Biggr] f(j),\quad k\in {\mathbb Z},
\end{align}
is well-defined and $x(\cdot)$ is a solution of \eqref{zajeb-equ}.
\item[(ii)] If the sequence $k\mapsto f(k),$ $k\in {\mathbb Z}$ is Bohr almost periodic and, for every relatively compact set $B\subseteq Y,$ $\kappa \in \circledast_{Y}$ and $\epsilon>0,$  there exists a finite real number
$L>0$ such that for each $t_{0}\in {\mathbb Z}$ there exists $\tau \in B(t_{0},L)\cap {\mathbb Z}$ such that
\begin{align}\label{mash}
\kappa \Bigl(  A(k+\tau)x-A(k)x\Bigr)\leq \epsilon,\quad k\in {\mathbb Z},\ x\in B,
\end{align}
then a unique almost periodic solution $x(\cdot)$ of \eqref{zajeb-equ} is given by \eqref{mileva}.
\item[(iii)] If $\omega \in {\mathbb N},$ $c\in {\mathbb C}\setminus \{0\},$ the sequence $k\mapsto f(k),$ $k\in {\mathbb Z}$ is $(\omega,c)$-periodic and $A(k+\omega)=A(k)$ for all $k\in {\mathbb Z},$ 
then an $(\omega,c)$-periodic solution $x(\cdot)$ of \eqref{zajeb-equ} is given by \eqref{mileva}. The uniqueness of  $(\omega,c)$-periodic solutions of \eqref{zajeb-equ} holds if $|c|\geq 1.$
\item[(iv)] If the sequence $k\mapsto f(k),$ $k\in {\mathbb Z}$ is bounded and belongs to the class $B-e-W_{ap}^{1}({\mathbb Z}: Y)$ as well as, for every relatively compact set $B\subseteq Y,$ $\kappa \in \circledast_{Y}$ and $\epsilon>0,$  there exist a  trigonometric polynomial $P:  {\mathbb Z} \rightarrow L(Y)$ and an integer $l\in {\mathbb N}$ such that 
\begin{align*}
l^{-1}\sum_{j=s}^{s+l} \Bigl[\kappa\bigl(A(j)x-P(j)x\bigr)\Bigr] \leq \epsilon ,\quad s\in {\mathbb Z},\ x\in B,
\end{align*}
then the sequence $x(\cdot)$, given by \eqref{mileva}, is a bounded solution  of \eqref{zajeb-equ} and belongs to the class $B-e-W_{ap}^{1}({\mathbb Z}: Y).$ 
\item[(v)] If the sequence $k\mapsto f(k),$ $k\in {\mathbb Z}$ is bounded and Besicovitch-$1$-almost periodic as well as, for every relatively compact set $B\subseteq Y,$ $\kappa \in \circledast_{Y}$ and $\epsilon>0,$  there exists a  trigonometric polynomial $P:  {\mathbb Z} \rightarrow L(Y)$ such that 
\begin{align}\label{nvid}
\limsup_{l\rightarrow +\infty}\Biggl\{ l^{-1}\cdot \sup_{x\in B}\sum_{j=-l}^{l}\Bigl[\kappa\bigl( A(j)x-P(j)x\bigr)\Bigr]\Biggr\} \leq \epsilon ,
\end{align}
then the sequence $x(\cdot)$, given by \eqref{mileva}, is a bounded, Besicovitch-$1$-almost periodic solution of \eqref{zajeb-equ}. 
\end{itemize}
\end{thm}

\begin{proof}
It can be simply proved that the series  $ \sum_{j=-\infty}^{k-2}[ A(k-1)\cdot ...\cdot A(j+1)] f(j)$ is absolutely convergent since for each $\kappa \in \circledast_{Y}$ we have
\begin{align*}
\kappa &\Biggl(\sum_{j=-\infty}^{k-2}\Bigl[ A(k-1)\cdot ...\cdot A(j+1)\Bigr] f(j) \Biggr)\leq \sum_{j=-\infty}^{k-2}\Bigl[ c^{\kappa}(j+1)\cdot ... \cdot c^{\kappa}(k-1)\Bigr]\kappa \bigl( f(j)\bigr)
\\& =\sum_{v=1}^{+\infty}\Bigl[ c^{\kappa}(k-+v)\cdot ... \cdot c^{\kappa}(k-1)\Bigr]\kappa \bigl( f(k-1-v)\bigr);
\end{align*}
cf. also \eqref{rac}.
Now it is quite simple to prove that the function $x(\cdot),$ given by \eqref{mileva}, is a solution of \eqref{zajeb-equ}: if $k\in {\mathbb Z},$ then we have
\begin{align*}
x(k+1)&=f(k)+\sum_{j=-\infty}^{k-1}\Bigl[ A(k)\cdot ...\cdot A(j+1)\Bigr] f(j) 
\\& =A(k)f(k-1)+A(k)\sum_{j=-\infty}^{k-2}\Bigl[ A(k-1)\cdot ...\cdot A(j+1)\Bigr] f(j) +f(k)
\\&=A(k)x(k)+f(k),
\end{align*}
which completes the proof of (i). In order to prove (ii), observe first that the range $B:=R(f)$ is relatively compact in $Y.$ Since the sequence $f(\cdot-1)$ is Bohr almost periodic and 
$$
x(k)=f(k-1)+\sum_{v=1}^{+\infty}\Bigl[A(k-1)\cdot ... \cdot A(k-v)\Bigr] \cdot f(k-1-v),\quad k\in {\mathbb Z},
$$
where the above series converges uniformly due to the Weierstrass criterion and condition \eqref{rac}, as easily shown, Lemma \ref{bochner1233210}(i) implies that it suffices to show that for each fixed integer $v\in {\mathbb N}$ the sequence $F_{v}(\cdot),$ given by
$$ 
F_{v}(k):=\bigl[A(k-1)\cdot ... \cdot A(k-v)\bigr]\cdot f(k-1-v),\quad k\in {\mathbb Z},
$$ 
is Bohr almost periodic. In order to that, we can apply the mathematical induction and the fact that the sequence $F_{1}(\cdot)$ is Bohr almost periodic. To show this, observe first that the prescribed assumptions imply that the function $G: {\mathbb Z} \rightarrow l_{\infty}(B: Y),$ given by $[G(k)](x):=A(k)x,$ $k\in {\mathbb Z},$ $x\in B$ is Bohr almost periodic. After that, we can apply Lemma \ref{bochner1233210}(ii) to see that, for every $\kappa \in \circledast_{Y}$ and $\epsilon>0,$  there exists a finite real number
$L>0$ such that for each $t_{0}\in {\mathbb Z}$ there exists $\tau \in B(t_{0},L)\cap {\mathbb Z}$ such that \eqref{mash} holds and
$$
\kappa \Bigl(  f(k+\tau)-f(k)\Bigr)\leq \epsilon,\quad k\in {\mathbb Z}. 
$$
Then the required conclusion follows form the estimates:
\begin{align*}
& \kappa \bigl( F_{1}(k+\tau)-F_{1}(k)\bigr)
\\& \leq c^{\kappa}(k-1+\tau)\kappa \bigl( f(k-2+\tau)-f(k-2) \bigr)+\sup_{x\in B}\kappa \bigl( A(k-1+\tau)x-A(k-1)x\bigr)
\\& \leq \Bigl[\sup_{s\in {\mathbb Z}}c^{\kappa}(s)\Bigr] \cdot \kappa \bigl( f(k-2+\tau)-f(k-2) \bigr)+\sup_{x\in B;s\in {\mathbb Z}}\kappa \bigl( A(s+\tau)x-A(s)x\bigr),
\end{align*}
for any $k\in {\mathbb Z}.$ To complete the proof of (ii), we only need to prove the uniqueness of Bohr almost periodic solutions of \eqref{zajeb-equ}. If $x(\cdot)$ is Bohr almost periodic and solves \eqref{zajeb-equ} with $f(\cdot)\equiv 0,$ then for each $k\in {\mathbb Z}$ we have $v(k+1)=A(-k-1)v(k)$, where $v(\cdot):=x(-\cdot),$ and therefore
$
v(k)=[A(-k)\cdot ... \cdot A(-1)]\cdot v(0),$ $k\in {\mathbb N}.$
Let $\kappa \in \circledast_{Y}$ and $\epsilon>0$ be fixed. Using the fact that
\begin{align}\label{rekaoqq}
\lim_{k\rightarrow +\infty}\Bigl[c^{\kappa}(-1)\cdot ...\cdot c^{\kappa}(-k)\Bigr]=0,
\end{align}
which immediately follows from \eqref{rac} with $k=0,$ we get the existence of an integer $k_{0}\in {\mathbb N}$ such that $\kappa(v(k))\leq \epsilon$ for $k\geq k_{0}.$ Since the number $\epsilon>0$ was arbitrary, the supremum formula established in \cite[Proposition 2.5]{vlad} yields that $\kappa(v(k))=0$ for each $k\in {\mathbb N}_{0}$. 
Now, since the seminorm $\kappa(\cdot)$ was arbitrary, the above yields $v(k)=0$ for all $k\in {\mathbb N}_{0}$ so that $u(k)=0$ for all $k\in -{\mathbb N} _{0}.$ This clearly implies $x(1)=A(0)x(0)=A(0)0=0$ and inductively $x(k)=0$ for all $k\in {\mathbb N},$ which completes the proof of (ii). Suppose now that $\omega \in {\mathbb N},$ $c\in {\mathbb C}\setminus \{0\},$ the sequence $k\mapsto f(k),$ $k\in {\mathbb Z}$ is $(\omega,c)$-periodic and $A(k+\omega)=A(k)$ for all $k\in {\mathbb Z}.$ Then it is clear that the sequence $x(\cdot),$ given by \eqref{mileva}, is $(\omega,c)$-periodic since
\begin{align*}
x(k+\omega)&=f(k-1+\omega)+\sum_{j=-\infty}^{k+\omega-2}\Biggl[ A(k-1+\omega)\cdot ...\cdot A(j+1)\Biggr] f(j)
\\& =cf(k-1)+\sum_{v=1}^{+\infty}\Biggl[ A(k-1+\omega)\cdot ...\cdot A(k-v+\omega)\Biggr] f(k-v+\omega-1)
\\& =cf(k-1)+c\sum_{v=1}^{+\infty}\Biggl[ A(k-1)\cdot ...\cdot A(k-v)\Biggr] f(k-v-1)
\\&=cf(k-1)+c\sum_{j=-\infty}^{k-2}\Biggl[ A(k-1)\cdot ...\cdot A(j+1)\Biggr] f(j)=cx(k),\quad k\in {\mathbb Z}.
\end{align*}
Suppose now that $|c|\geq 1$ and $x(k+1)=A(k)x(k)$, $k\in {\mathbb Z}.$ Since $A(k+\omega)=A(k)$, $k\in {\mathbb Z},$ we may assume without loss of generality that for each $\kappa \in \circledast_{Y}$ we have $c^{\kappa}(k+\omega)=c^{\kappa}(k),$ $k\in {\mathbb Z}.$ Then $\kappa (x(k+1))\leq c^{\kappa}(k)\kappa (x(k))$, $k\in {\mathbb Z},$ $\kappa \in \circledast_{Y},$ which inductively implies that 
\begin{align*}
|c|\kappa (x(k))=\kappa (x(k+\omega))\leq c^{\kappa}(k+\omega-1)\cdot ...\cdot c^{\kappa}(k)\kappa (x(k)), \quad k\in {\mathbb Z},\ \kappa \in \circledast_{Y}.
\end{align*}
If we assume that $x(\cdot)$ is not equal to the zero sequence, then the above simply implies the existence of a seminorm $\kappa \in \circledast_{Y}$ and  an integer $k_{0}\in [1,\omega]$ such that, for every $l\in {\mathbb Z},$ we have
\begin{align*}
|c|\leq c^{\kappa}\bigl(k_{0}+(1+l)\omega-1 \bigr)\cdot ...\cdot c^{\kappa}\bigl(k_{0}+l\omega \bigr), \quad l\in {\mathbb Z}.
\end{align*}
On the other hand, \eqref{rekaoqq} implies $\lim_{l\rightarrow+\infty}[c^{\kappa}(k_{0})\cdot ... \cdot c^{\kappa}(k_{0}-1-l\omega)]=0,$ which is a contradiction since $|c|\geq 1$  and $c^{\kappa}(k+\omega)=c^{\kappa}(k),$ $k\in {\mathbb Z}.$
This completes the proof of (iii). 

The proofs of parts (iv) and (v) are similar to proof of (ii); we will only outline the main details of  proof for (v). The collection of all bounded, Besicovitch-$1$-almost periodic sequence from ${\mathbb Z}$ into $Y$, denoted by $B_{ap}({\mathbb Z}: Y),$ is a vector space with the usal operations and $B_{ap}({\mathbb Z}: Y)$ is translation invariant; moreover, if  $u_{v} :{\mathbb N} \rightarrow Y$ is a bounded, Besicovitch-$1$-almost periodic sequence for each $v\in {\mathbb N}$ and $\sum_{v=1}^{\infty}u_{v}(k)=u(k)$ for each $k\in {\mathbb Z},$ where the above series is absolutely convergent, then the sequence 
$u :{\mathbb Z} \rightarrow Y$ is a bounded and Besicovitch-$1$-almost periodic, as well. To complete the proof, it remains to be proved that the sequence $F_{v} :{\mathbb Z} \rightarrow Y$ defined in part (ii) is bounded and Besicovitch-$1$-almost periodic for all $v\in {\mathbb N}.$ Let $\epsilon>0$ and $\kappa \in \circledast_{Y}$ be fixed. If $v=1,$ then we know that there exists a trigonometric polynomial $P_{0}:  {\mathbb Z}\rightarrow Y$ such that
\begin{align}\label{brest}
\limsup_{l\rightarrow +\infty}\Biggl\{ l^{-1}\cdot \sup_{x\in B}\sum_{j=-l}^{l}\Bigl[\kappa\Bigl( F(j)-P_{0}(j)\Bigr)\Bigr]\Biggr\}<\epsilon.
\end{align}
The set $B=R(P)$ is relatively compact in $Y$ and the prescribed assumption implies that there exists a trigonometric polynomial $P:  {\mathbb Z} \rightarrow L(Y)$ such that \eqref{nvid} holds. To deduce the required conclusion, we can simply use the estimates \eqref{nvid}, \eqref{brest} and decomposition
\begin{align*}
& A(j-1)f(j-2)-P(j-1)P_{0}(j-2)
\\&=\Bigl[ A(j-1)-P(j-1)\Bigr]P_{0}(j-2)+A(j-1)\Bigl[ f(j-2)-P(j-2)\Bigr],\quad j\in {\mathbb Z};
\end{align*}
if $v=2,$ then we can use the already established conclusion with $v=1$ and decomposition
\begin{align*}
& A(j-1)A(j-2)f(j-3)-P(j-1)-P(j-2)P_{0}(j-3)
\\&=\Bigl[ A(j-1)-P(j-1)\Bigr]P(j-2)P_{0}(j-3)
\\&+A(j-1)\Bigl[ A(j-2)f(j-3)-P(j-2)P_{0}(j-3)\Bigr],\quad j\in {\mathbb Z}.
\end{align*}
After that, we can proceed inductively and show that the sequence $F_{v}: {\mathbb Z} \rightarrow Y$ belongs to the space $B_{ap}({\mathbb Z}: Y)$ for each $v\in {\mathbb N}.$
\end{proof}

\begin{rem}\label{ogr}
\begin{itemize}
\item[(i)]
Suppose, in place of condition \eqref{rac}, that there exists $\alpha>0$ such that the set $\{(1+|k|)^{-\alpha}f(k)  :k\in {\mathbb Z}\}$ is bounded in $Y$ and
\begin{align*}
\sum_{v=1}^{+\infty}\Bigl[ c^{\kappa}(k-v)\cdot ... \cdot c^{\kappa}(k-1)\Bigr] \cdot v^{\alpha}<+\infty , \quad k\in {\mathbb Z}.
\end{align*}
Then the sequence $x(\cdot)$, given by \eqref{mileva}, is 
is well-defined, $x(\cdot)$ is a solution of \eqref{zajeb-equ} and the set $\{(1+|k|)^{-\alpha}x(k)  :k\in {\mathbb Z}\}$ is bounded in $Y$. We can also consider the exponentially bounded sequences $f(\cdot)$ in $Y.$
\item[(ii)] Condition \eqref{rac} particularly holds if for each $\kappa \in \circledast_{Y}$ there exists $c^{\kappa}\in (0,1)$ such that, for every $k\in {\mathbb Z},$ one has $c^{\kappa}(k) \leq c^{\kappa}.$
\item[(ii)] The uniqueness of $(\omega,c)$-periodic solutions of \eqref{zajeb-equ} does not hold if $|c|<1.$ Consider, for example, the case in which $\omega=1,$ $c=1/2$, $f(\cdot)\equiv 0,$ $A(k):=(1/2){\rm I},$ $k\in {\mathbb Z}$ and $x(k):=(1/2^{k})x,$ $k\in {\mathbb Z},$ where $x\in Y$ and $x\neq 0.$
\end{itemize}
\end{rem}

Suppose now that $C\in L(Y)$ and for each integer $k\in {\mathbb Z}$ we have that ${\mathcal A}(k)$ is an MLO in $Y$ and $[{\mathcal A}(k)]^{-1}C\in L(Y).$ Consider now the abstract nonautonomous difference inclusion
\begin{align}\label{zajeb}
Cx(k+1)\in {\mathcal A}(k)x(k)+Cf(k),\quad k\in {\mathbb Z}.
\end{align}
Under our assumptions, we have that \eqref{zajeb} is equivalent with $x(k)=[{\mathcal A}(k)]^{-1}C[x(k+1)-f(k)],$ $k\in {\mathbb Z},$ i.e., with 
\begin{align}\label{ztrans}
x(k)=\bigl[{\mathcal A}(k)\bigr]^{-1}Cx(k+1)-\bigl[{\mathcal A}(k)\bigr]^{-1}Cf(k),\quad k\in {\mathbb Z}.
\end{align}
Using the substitution $k\mapsto -k,$ $k\in {\mathbb Z},$ it readily follows that \eqref{ztrans} is equivalent with 
\begin{align*}
x(-k)=\bigl[{\mathcal A}(-k)\bigr]^{-1}Cx(-k+1)-\bigl[{\mathcal A}(-k)\bigr]^{-1}Cf(-k),\quad k\in {\mathbb Z},
\end{align*}
i.e., with
\begin{align*}
v(k)=\bigl[{\mathcal A}(-k)\bigr]^{-1}Cv(k-1)-\bigl[{\mathcal A}(-k)\bigr]^{-1}Cf(-k),\quad k\in {\mathbb Z},
\end{align*}
which is clearly equivalent with
\begin{align}\label{ztrans1}
v(k+1)=\bigl[{\mathcal A}(-k-1)\bigr]^{-1}Cv(k)-\bigl[{\mathcal A}(-k-1)\bigr]^{-1}Cf(-k-1),\quad k\in {\mathbb Z},
\end{align}
where $v(k):=u(-k),$ $k\in {\mathbb Z}.$ 
Using now the elementary argumentation involving Theorem \ref{zajeb-e} and the equivalence of \eqref{zajeb} and \eqref{ztrans1}, we immediately get the following result (the injectivity of regularizing operator $C$ is not required here):

\begin{thm}\label{zajeb-mlo}
Suppose that $C\in L(Y)$ and for each integer $k\in {\mathbb Z}$ we have that ${\mathcal A}(k)$ is an \emph{MLO} in $Y$ and $[{\mathcal A}(k)]^{-1}C\in L(Y).$
Suppose, further, that
for each $k\in {\mathbb Z}$ and $\kappa \in \circledast_{Y}$ there exists a finite real number $c^{\kappa}(k)>0$ such that $\kappa([{\mathcal A}(k)]^{-1}Cx)\leq c^{\kappa}(k)\kappa(x)$ for all $x\in Y$, \eqref{rac} holds,
and for each $\kappa \in \circledast_{Y}$ we have $\sup_{k\in {\mathbb Z}}c^{\kappa}(k)<+\infty .$ 
Then we have the following:
\begin{itemize}
\item[(i)] If the sequence $f : {\mathbb Z} \rightarrow Y$ is bounded, then the sequence $x(\cdot),$ given by
\begin{align}\label{milevai}
x(k):=-\bigl[{\mathcal A}(k)\bigr]^{-1}Cf(k)-\sum_{j=-\infty}^{k-2}\Biggl[ \bigl[{\mathcal A}(k)\bigr]^{-1}C\cdot ...\cdot \bigl[{\mathcal A}(-j-1)\bigr]^{-1}C\Biggr] f(-j-1),
\end{align}
for any $k\in {\mathbb Z},$
is well-defined and $x(\cdot)$ is a solution of \eqref{zajeb}.
\item[(ii)] If the sequence $k\mapsto [{\mathcal A}(k)]^{-1}Cf(k),$ $k\in {\mathbb Z}$ is Bohr almost periodic and, for every relatively compact set $B\subseteq Y,$ $\kappa \in \circledast_{Y}$ and $\epsilon>0,$  there exists a finite real number
$L>0$ such that for each $t_{0}\in {\mathbb Z}$ there exists $\tau \in B(t_{0},L)\cap {\mathbb Z}$ such that
$$
\kappa \Bigl(  \bigl[{\mathcal A}(k+\tau)\bigr]^{-1}Cx-\bigl[{\mathcal A}(k)\bigr]^{-1}Cx\Bigr)\leq \epsilon,\quad k\in {\mathbb Z},\ x\in B,
$$
then a unique almost periodic solution $x(\cdot)$ of \eqref{zajeb} is given by \eqref{milevai}.
\item[(iii)] If $\omega \in {\mathbb N},$ $c\in {\mathbb C}\setminus \{0\},$ the sequence $k\mapsto f(k),$ $k\in {\mathbb Z}$ is $(\omega,c)$-periodic and ${\mathcal A}(k+\omega)={\mathcal A}(k)$ for all $k\in {\mathbb Z},$ 
then an $(\omega,c)$-periodic solution $x(\cdot)$ of \eqref{zajeb} is given by \eqref{milevai}. The uniqueness of  $(\omega,c)$-periodic solutions of \eqref{zajeb} holds if $|c|\geq 1.$
\item[(iv)] If the sequence $k\mapsto [{\mathcal A}(k)]^{-1}Cf(k),$ $k\in {\mathbb Z}$ is bounded and belongs to the class $B-e-W_{ap}^{1}({\mathbb Z}: Y)$ as well as, for every relatively compact set $B\subseteq Y,$ $\kappa \in \circledast_{Y}$ and $\epsilon>0,$ there exist a  trigonometric polynomial $P:  {\mathbb Z} \rightarrow L(Y)$ and an integer $l\in {\mathbb N}$ such that 
\begin{align*}
l^{-1}\sum_{j=s}^{s+l} \Bigl[\kappa\Bigl(\bigl[{\mathcal A}(j)\bigr]^{-1}Cx-P(j)x\Bigr)\Bigr] \leq \epsilon ,\quad s\in {\mathbb Z},\ x\in B,
\end{align*}
then the sequence $x(\cdot)$, given by \eqref{milevai}, is a bounded solution  of \eqref{zajeb} and belongs to the class $B-e-W_{ap}^{1}({\mathbb Z}: Y).$ 
\item[(v)] If the sequence $k\mapsto [{\mathcal A}(k)]^{-1}Cf(k),$ $k\in {\mathbb Z}$ is bounded and Besicovitch-$1$-almost periodic as well as, for every relatively compact set $B\subseteq Y,$ $\kappa \in \circledast_{Y}$ and $\epsilon>0,$  there exists a  trigonometric polynomial $P:  {\mathbb Z} \rightarrow L(Y)$ such that 
\begin{align*}
\limsup_{l\rightarrow +\infty}\Biggl\{ l^{-1}\cdot \sup_{x\in B}\sum_{j=-l}^{l}\Bigl[\kappa\Bigl( \bigl[{\mathcal A}(j)\bigr]^{-1}Cx-P(j)x\Bigr)\Bigr]\Biggr\} \leq \epsilon ,
\end{align*}
then the sequence $x(\cdot)$, given by \eqref{milevai}, is a bounded, Besicovitch-$1$-almost periodic solution of \eqref{zajeb}. 
\end{itemize}
\end{thm} 

\begin{rem}\label{ghj}
Suppose that, for every $k\in {\mathbb Z},$ there exists a linear continuous operator $D(k)\in L(Y)$ such that $D(k)\subseteq [{\mathcal A}(k)]^{-1}C$ and all remaining assumptions given in the formulation of Theorem \ref{zajeb-mlo} hold true with the operators $[{\mathcal A}(\cdot)]^{-1}C$ replaced by the operators $D(\cdot)$ therein. If the sequence $f : {\mathbb Z} \rightarrow Y$ is bounded, then the sequence $x(\cdot),$ given by replacing the operators $[{\mathcal A}(\cdot)]^{-1}C$ by the operators $D(\cdot)$ in \eqref{milevai},
is well-defined and $x(\cdot)$ is a solution of \eqref{zajeb}. Similarly, if we replace the operators $[{\mathcal A}(\cdot)]^{-1}C$ by the operators $D(\cdot)$ in \eqref{milevai} and all other appearances in parts (ii)-(v), then the conclusions clarified in (iv) and (v) continue to hold without any changes; in (ii), resp. (iii), we can only prove the existence of an almost periodic solution $x(\cdot)$ of \eqref{zajeb}, resp. an $(\omega,c)$-periodic solution $x(\cdot)$ of \eqref{zajeb}, for any $c\in {\mathbb C}\setminus \{0\}$.
\end{rem}

We will illustrate Theorem \ref{zajeb-mlo} with the following example:

\begin{example}\label{eel}
Let $E$ be $L^p(\mathbb R^n)$ ($1\leq p\leq\infty)$, $C_0(\mathbb R^n)$, $C_b(\mathbb R^n)$ or $BUC(\mathbb R^n),$ and let $0\leq l\leq n.$ 
The space $E_l$ is defined by $ E_l:=\{f\in E:f^{(\alpha)}\in E\text{ for all }\alpha \in {\mathbb N}_{0,l}\},$ where ${\mathbb N}_{0,l}:=\{\alpha \in {{\mathbb N_0^n}} : \alpha_{l+1}=...=\alpha_{n}=0\}$;
the family of seminorms\index{$E_l$-type spaces} $(q_{\alpha}(f):=||f^{(\alpha)}||_E,\ f\in E_l;\ \alpha\in {\mathbb N}_{0,l})$ induces a Fr\'echet topology on $E_l$. Let 
$\omega \in {\mathbb N},$ $a_{\eta , j}\in\mathbb C$, $0\leq|\eta|\leq N_{j}$, and let  the operator $B(j)f:=P_{j}(D)f\equiv\sum_{|\eta|\leq N_{j}}a_{\eta ,j}f^{(\eta)}$ act with its maximal distributional domain in $E_{l}$ for $1\leq j\leq \omega;$  
then $B(j)$ is a closed linear operator on $E_l$ for $1\leq j\leq \omega .$ Suppose, for the sake of brevity, that $\delta <0$ and
\begin{equation*}
\sup_{x\in\mathbb R^n}\Re\Biggl ( \sum_{|\eta|\leq N_{j}}a_{\eta,j}i^{|\eta|}x^{\eta}\Biggr)\leq\delta,\quad 1\leq j\leq \omega.
\end{equation*}
Further on,
let $r_{p}=n|(1/2)-(1/p)|$ if $E=L^p(\mathbb R^n)$ for some $p\in (1,\infty),$ and let $r\geq r_{p}$ in this case; otherwise, we assume that $r>n/2$. Let the continuous linear operator $C_{r,l,j}\in L(E_{l})$ has the same meaning as in the formulation of \cite[Theorem 5.10, p. 32]{x263}.
This result implies that $B(j)$ generates an exponentially equicontinuous $C_{r,l,j}$-regularized semigroup $(T_{r,l,j}(t))_{t\geq 0}$ on $E_{l}$ such that there exists a finite real constant $M_{j}\geq1$ such that, for every $\alpha\in {\mathbb N}_{0,l},$ we have $q_{\alpha}(T_{r,l,j}(t)f)\leq  M_{j}(1+t)^{r}e^{\delta t}q_{\alpha}(f),$ $t\geq 0 ,$ $f\in E_{l}$, $1\leq j\leq \omega$. If
$C= C_{r,l,1}\cdot ... \cdot C_{r,l,\omega}$, then there exists a finite real number $M>0$ such that $B(j)$ generates an exponentially equicontinuous $C$-regularized semigroup $(T_{j}(t))_{t\geq 0}$ on $E_{l}$ such that, for every $\alpha\in {\mathbb N}_{0,l},$ we have $q_{\alpha}(T_{j}(t)f)\leq  M(1+t)^{r}e^{\delta t}q_{\alpha}(f),$ $t\geq 0 ,$ $f\in E_{l}$, $1\leq j\leq \omega$. Summa summarum, there exist finite real numbers $\lambda \in (0,1)$ and $d_{1}>0, \ ...,\ d_{\omega}>0$ such that the operators $A(1):=d_{1}B(1),\ ...,\ A(\omega)=d_{\omega}B(\omega)$ satisfy 
$[A(j)]^{-1}C\in L(E_{l})$ and $q_{\alpha}([A(j)]^{-1}Cf)\leq \lambda q_{\alpha}(f),$ $f\in E_{l}$, $1\leq j\leq \omega$, $\alpha\in {\mathbb N}_{0,l}$. After that, we extend the sequence $A(\cdot)$ periodically to the whole integer line so that $A(k+\omega)=A(k)$, $k\in {\mathbb Z}.$  If $c\in {\mathbb C}\setminus \{0\},$ the sequence $k\mapsto f(k),$ $k\in {\mathbb Z}$ is $(\omega,c)$-periodic,
then an $(\omega,c)$-periodic solution $x(\cdot)$ of \eqref{zajeb}, with ${\mathcal A}(k)=A(k)$ for all $k\in {\mathbb Z}$, is given by \eqref{milevai}. The uniqueness of $(\omega,c)$-periodic solutions of \eqref{zajeb} holds if $|c|\geq 1.$

Here we can also use the projective limit of Fr\' echet spaces $Y:=\bigcap_{1\leq p<\infty}E_{l}^{p}$, which is a complete locally convex space due to \cite[Proposition 24.4]{meise}; it is also worth recalling that the resolvent set of all operators $A(k)$, $k\in {\mathbb Z}$ can be empty if the polynomials $P_{1}(\cdot),\ ...,\ P_{\omega}(\cdot)$ are not coercive (cf. \cite{abas-vlade} for more details).
\end{example}

Consider now the abstract difference equation
\begin{align}\label{vb}
CB(k+1)u(k+1)=A(k)u(k)+Cf(k),\quad k\in {\mathbb Z}
\end{align}
and  the abstract difference equation
\begin{align}\label{vb1}
B(k+1)Cu(k+1)=A(k)u(k)+Cg(k),\quad k\in {\mathbb Z},
\end{align}
where $C\in L(Y)$, $A(k)$ and $B(k)$ are linear operators on $Y$ ($k\in {\mathbb Z}$), $f: {\mathbb Z} \rightarrow Y$ and $g: {\mathbb Z} \rightarrow Y$ are given sequences. In order to analyze the equation \eqref{vb}, we introduce the substitution $v(k):=B(k)u(k),$ $k\in {\mathbb Z}.$ Then it can be simply shown that the equation \eqref{vb} is equivalent with the inclusion
\begin{align}\label{vbm}
Cv(k+1)\in \Bigl[A(k)\bigl[ B(k)\bigr]^{-1}\Bigr]v(k)+Cf(k),\quad k\in {\mathbb Z}.
\end{align}
On the other hand, it is clear that \eqref{vb1} is equivalent with the inclusion
\begin{align}\label{vb1m}
Cu(k+1)\in \Bigl[\bigl[ B(k+1)\bigr]^{-1}A(k)\Bigr]u(k)+Cf(k),\quad k\in {\mathbb Z},
\end{align}
where $Cf(k)=[ B(k+1)]^{-1}Cg(k)$ for all $k\in {\mathbb Z}.$
Applying Theorem \ref{zajeb-mlo} to \eqref{vbm} and \eqref{vb1m}, we immediately get the following results (we can similarly analyze the existence and uniqueness of Weyl and Besicovitch almost periodic type solutions to \eqref{vb} and \eqref{vb1}; Remark \ref{ghj} can be also simply rephrased, which will be important in Example \ref{karin} below):

\begin{thm}\label{zajeb-mlo1}
Suppose that $C\in L(Y)$, $f: {\mathbb Z} \rightarrow Y$ and for each integer $k\in {\mathbb Z}$ we have that $A(k)$ and $B(k)$ are linear operators on $Y$ such that $B(k)[A(k)]^{-1}C\in L(Y).$
Suppose, further, that
for each $k\in {\mathbb Z}$ and $\kappa \in \circledast_{Y}$ there exists a finite real number $c^{\kappa}(k)>0$ such that $\kappa(B(k)[A(k)]^{-1}Cx)\leq c^{\kappa}(k)\kappa(x)$ for all $x\in Y$, \eqref{rac} holds,
and for each $\kappa \in \circledast_{Y}$ we have $\sup_{k\in {\mathbb Z}}c^{\kappa}(k)<+\infty .$ 
Then we have the following:
\begin{itemize}
\item[(i)] If the sequence $f : {\mathbb Z} \rightarrow Y$ is bounded, then the sequence $v(\cdot),$ given by
\begin{align}
\notag v(k):&=-B(k)\bigl[A(k)\bigr]^{-1}Cf(k)
\\\label{milevaim} &-\sum_{j=-\infty}^{k-2}\Biggl[ \Bigl[ B(k)\bigl[A(k)\bigr]^{-1}C\Bigr]\cdot ...\cdot \Bigl[ B(-j-1)\bigl[A(-j-1)\bigr]^{-1}C\Bigr]\Biggr] f(-j-1),
\end{align}
for any $k\in {\mathbb Z},$
is well-defined and $v(\cdot)$ is a solution of \eqref{vbm}. Furthermore, the sequence $B(\cdot)v(\cdot)$ is a solution of \eqref{vb}.
\item[(ii)] If the sequence $k\mapsto B(k)[A(k)]^{-1}Cf(k),$ $k\in {\mathbb Z}$ is Bohr almost periodic and, for every relatively compact set $B\subseteq Y,$ $\kappa \in \circledast_{Y}$ and $\epsilon>0,$  there exists a finite real number
$L>0$ such that for each $t_{0}\in {\mathbb Z}$ there exists $\tau \in B(t_{0},L)\cap {\mathbb Z}$ such that
$$
\kappa \Bigl(  B(k+\tau)\bigl[A(k+\tau)\bigr]^{-1}Cx-B(k)\bigl[A(k)\bigr]^{-1}Cx\Bigr)\leq \epsilon,\quad k\in {\mathbb Z},\ x\in B,
$$
then a unique almost periodic solution $v(\cdot)$ of \eqref{vbm} is given by \eqref{milevaim}.  Furthermore, the sequences $A(\cdot)v(\cdot)$ and $B(\cdot)v(\cdot)$ are Bohr almost periodic and there exists a unique Bohr almost periodic solution of \eqref{vb}.
\item[(iii)] If $\omega \in {\mathbb N},$ $c\in {\mathbb C}\setminus \{0\},$ the sequence $k\mapsto f(k),$ $k\in {\mathbb Z}$ is $(\omega,c)$-periodic, $A(k+\omega)=A(k)$ and $B(k+\omega)=B(k)$ for all $k\in {\mathbb Z},$ 
then an $(\omega,c)$-periodic solution $v(\cdot)$ of \eqref{vbm} is given by \eqref{milevaim} and an $(\omega,c)$-periodic solution of \eqref{vb} is given by $B(\cdot)v(\cdot)$. The uniqueness of  $(\omega,c)$-periodic solutions to \eqref{vbm} and \eqref{vb} holds if $|c|\geq 1.$
\end{itemize}
\end{thm}

\begin{thm}\label{zajeb-mlo2}
Suppose that $C\in L(Y)$, $f: {\mathbb Z} \rightarrow Y,$ $g: {\mathbb Z} \rightarrow Y$ and for each integer $k\in {\mathbb Z}$ we have that $A(k)$ and $B(k)$ are linear operators on $Y$ such that $[A(k)]^{-1}B(k+1)C\in L(Y)$  and $[B(k+1)]^{-1}Cg(k)=Cf(k)$.
Suppose, further, that
for each $k\in {\mathbb Z}$ and $\kappa \in \circledast_{Y}$ there exists a finite real number $c^{\kappa}(k)>0$ such that $\kappa([A(k)]^{-1}B(k+1)Cx)\leq c^{\kappa}(k)\kappa(x)$ for all $x\in Y$, \eqref{rac} holds,
and for each $\kappa \in \circledast_{Y}$ we have $\sup_{k\in {\mathbb Z}}c^{\kappa}(k)<+\infty .$ 
Then we have the following:
\begin{itemize}
\item[(i)] If the sequence $f : {\mathbb Z} \rightarrow Y$ is bounded, then the sequence $x(\cdot),$ given by
\begin{align}
\notag u(k):&=-\bigl[A(k)\bigr]^{-1}B(k+1)Cf(k)
\\& \label{milevaid}-\sum_{j=-\infty}^{k-2}\Biggl[ \Bigl[\bigl[A(k)\bigr]^{-1}B(k+1)C\Bigr]\cdot ...\cdot  \Bigl[\bigl[A(-j-1)\bigr]^{-1}B(-j)C\Bigr]\Biggr] f(-j-1),
\end{align}
for any $k\in {\mathbb Z},$
is well-defined and $u(\cdot)$ is a solution of \eqref{vb1}.
\item[(ii)] If the sequence $k\mapsto [A(k)]^{-1}B(k+1)Cf(k),$ $k\in {\mathbb Z}$ is Bohr almost periodic and, for every relatively compact set $B\subseteq Y,$ $\kappa \in \circledast_{Y}$ and $\epsilon>0,$  there exists a finite real number
$L>0$ such that for each $t_{0}\in {\mathbb Z}$ there exists $\tau \in B(t_{0},L)\cap {\mathbb Z}$ such that
$$
\kappa \Bigl(  \bigl[A(k+\tau)\bigr]^{-1}B(k+1+\tau)Cx- \bigl[A(k)\bigr]^{-1}B(k+1)Cx\Bigr)\leq \epsilon,\quad k\in {\mathbb Z},\ x\in B,
$$
then a unique almost periodic solution $u(\cdot)$ of \eqref{vb1} is given by \eqref{milevaid}.
\item[(iii)] If $\omega \in {\mathbb N},$ $c\in {\mathbb C}\setminus \{0\},$ the sequence $k\mapsto f(k),$ $k\in {\mathbb Z}$ is $(\omega,c)$-periodic, $A(k+\omega)=A(k)$ and $B(k+\omega)=B(k)$ for all $k\in {\mathbb Z},$ 
then an $(\omega,c)$-periodic solution $u(\cdot)$ of \eqref{vb1} is given by \eqref{milevaid}. The uniqueness of  $(\omega,c)$-periodic solutions of \eqref{vb1} holds if $|c|\geq 1.$
\end{itemize}
\end{thm}

Concerning the condition $B(k)[A(k)]^{-1}C\in L(Y)$ clarified in the formulation of Theorem \ref{zajeb-mlo1}, we have the following observation:

\begin{rem}\label{zat}
Suppose that $C\in L(Y)$, $A(k)$ is a linear operator on $Y$ and $B(k)$ is a closed linear operator on $Y$ ($k\in {\mathbb Z}$). Then $B(k)[A(k)]^{-1}C\in L(Y)$ provided that $[A(k)]^{-1}C\in L(Y)$  and
$D(A(k))\subseteq D(B(k))$ for all $k\in {\mathbb Z}$, as well as the space $Y$ is ultra-bornological and has a web; this follows from the fact that the operator $B(k)[A(k)]^{-1}C$ is closed and defined on the whole space $Y$ (cf. \cite[Closed graph theorem, p. 289]{meise}). 
\end{rem}

Now we will illustrate Theorem \ref{zajeb-mlo1} and Theorem \ref{zajeb-mlo2} with the following examples:

\begin{example}\label{rt}
Let $E$ be $L^p(\mathbb R^n)$ ($1\leq p\leq\infty)$, $C_0(\mathbb R^n)$, $C_b(\mathbb R^n)$ or $BUC(\mathbb R^n),$ and let $0\leq l\leq n;$ we will consider the Fr\' echet space $E_{l}$ from  Example \ref{eel}. Suppose that $\gamma \in (0,2),$ $\omega \geq 0$ and the operator $A$ generates an exponentially equicontinuous $(g_{\gamma},C)$-regularized $C$-resolvent family $(R_{\gamma}(t))_{t\geq 0}$ on $E_{l}$ such that there exists $M>0$ such that 
\begin{align}\label{tajp}
q_{\alpha}(R_{\gamma}(t)f)\leq Me^{\omega t}q_{\alpha}(f),\quad f\in E_{l},\ t\geq 0,\ \alpha \in {\mathbb N}_{0,l};
\end{align} 
cf. \cite[Theorem 2.2.21, Remark 2.2.23]{FKP} with $P_{2}(x)\equiv 1$ for some examples of the abstract differential operators $A$ with constant coefficients which obey the above properties. Let $\sigma>0,$ let a sequence $b: {\mathbb Z} \rightarrow (\omega^{\gamma}+\sigma,+\infty)$ be Bohr almost periodic and let $m : {\mathbb Z} \rightarrow {\mathbb C} $ be a Bohr almost periodic  sequence.
 Let us consider the operators $B_{0}(k):=[m(k)]{\rm I},$ $k\in {\mathbb Z}$ and $A(k):=b(k)-A$ for all $k\in {\mathbb Z}$. Then $B_{0}(k)\in L(E_{l})$ and $[A(k)]^{-1}C\in L(E_{l})$ for all $k\in {\mathbb Z}.$
Furthermore, 
$$
\lambda^{\gamma-1}\bigl( \lambda^{\gamma}-A \bigr)^{-1}Cf=\int^{\infty}_{0}e^{-\lambda t}R_{\gamma}(t)f\, dt,\quad f\in E_{l},\ \Re \lambda>\omega,
$$
which simply implies that there exists a finite real constant $M_{1}>0$ such that $q_{\alpha}(B_{0}(k)[A(k)]^{-1}Cf)\leq M_{1}[\sup_{k\in {\mathbb Z}}|m(k)| ]\cdot q_{\alpha}(f):=M_{1}m_{\infty}\cdot q_{\alpha}(f)$ for all $k\in {\mathbb Z}$, $f\in E_{l}$ and $\alpha \in {\mathbb N}_{0,l}.$ 
Then for each $\alpha \in {\mathbb N}_{0,l}$ we have $q_{\alpha}(B(k)[A(k)]^{-1}Cf)\leq M_{0}q_{\alpha,1}(f),$ $f\in E_{l},$ $k\in {\mathbb Z},$ 
where $M_{0}\in (0,1)$ and $B(k):=M_{0}B_{0}(k)/(1+M_{1}m_{\infty})$ for all $k\in {\mathbb Z}.$

We will prove that the requirements of Theorem \ref{zajeb-mlo1} hold with the operator $C$ replaced therein with the operator $C^{2},$ provided that the sequence $k\mapsto f(k),$ $k\in {\mathbb Z}$ is Bohr almost periodic; we can similarly analyze the possible application of Theorem \ref{zajeb-mlo2}. Toward this end, it is only worth noting that the use of $C$-resolvent equation  (cf. \cite[Theorem 1.2.4(ii)]{FKP}) implies that  there exists a finite real constant $M_{2}>0$ such that for each multi-index $\alpha \in {\mathbb N}_{0,l}$and $f\in E_{l}$ we have
\begin{align*} &
q_{\alpha,1}\Bigl(  B(k+\tau)\bigl[A(k+\tau) \bigr]^{-1}C^{2}f-B(k)\bigl[A(k) \bigr]^{-1}C^{2}f\Bigr) 
\\& \leq \bigl| m(k+\tau)-m(k)\bigr| q_{\alpha}\Bigl( \bigl(b(k+\tau)-A \bigr)^{1}C^{2}f \Bigr)
\\& +|m(k)|\cdot q_{\alpha}\Bigl(\bigl(b(k+\tau)-A \bigr)^{1}C^{2}f-\bigl(b(k)-A \bigr)^{1}C^{2}f\Bigr)
\\& \leq \bigl| m(k+\tau)-m(k)\bigr| q_{\alpha}\Bigl( \bigl(b(k+\tau)-A \bigr)^{1}C^{2}f \Bigr)
\\&+\bigl| b(k+\tau)-b(k)\bigr| q_{\alpha}\Bigl( \bigl(b(k+\tau)-A \bigr)^{1}C \cdot \bigl(b(k)-A \bigr)^{1}Cf \Bigr)
\\& \leq M_{2}\Bigl[\bigl| m(k+\tau)-m(k)\bigr| q_{\alpha}(f)+ \bigl| b(k+\tau)-b(k)\bigr| q_{\alpha}(f)\Bigr],\quad k\in {\mathbb Z},\ \tau \in {\mathbb Z}.
\end{align*}
In particular, there exists a unique Bohr almost periodic solution of the abstract difference equation \eqref{vb} with the operator $C$ replaced therein with the operator $C^{2}.$
Here we can also use the projective limit of Fr\' echet spaces $Y:=\bigcap_{1\leq p<\infty}E_{l}^{p}$, as in Example \ref{eel}.
\end{example}

\begin{example}\label{dl}
Assume that the set $\emptyset \neq \Omega \subseteq {\mathbb R}^{n}$ is open, bounded and Dirichlet regular. Then we know that the Dirichlet Laplacian $\Delta_{0}$ generates a positive contractive strongly continuous semigroup on the space $C_{0}(\Omega):=\{u\in C(\overline{\Omega}) : u_{| \partial \Omega}=0\},$ equipped with the sup-norm; cf. \cite[Chapter 6]{a43} for the notion and more details. Using Theorem \ref{zajeb-mlo1}, we can consider the existence and uniqueness of almost periodic solutions to 
the following abstract nonautonomous semi-discrete Poisson heat equation in the space $C_{0}(\Omega)$:
\[(P):\left\{
\begin{array}{l}
m(k+1,x)u(k+1,x)=\Delta_{0} u(k,x) -b(k)u(k,x)+f(k,x),\ k\in {\mathbb Z},\ x\in {\Omega};\\
u(k,x)=0,\quad (k,x)\in {\mathbb Z} \times \partial \Omega ,
\end{array}
\right.
\]
where the sequence $m : {\mathbb Z}\rightarrow C_{0}(\Omega)$ is almost periodic, the sequence $f : {\mathbb Z}\rightarrow C_{0}(\Omega)$ is almost periodic, the sequence $b : {\mathbb Z}\rightarrow {\mathbb C}$ is almost periodic, the value of term $\sup_{k\in {\mathbb Z};x\in \overline{\Omega}}|m(k,x)|$ is sufficiently small and there exists $c>0$ such that $\Re (b(k))>c$ for all $k\in {\mathbb Z}.$ Using Theorem \ref{zajeb-mlo1} with the operators $C=I$, $[B(k)f](x):=m(k,x)f(x),$ $k\in {\mathbb Z},$ $x\in \overline{\Omega},$ $f\in C_{0}(\Omega)$ and $A(k):=\Delta_{0}-b(k)$, $k\in {\mathbb Z},$ the computation carried out in the final part of Example \ref{rt} with the operator $A=\Delta_{0}$ shows that there exists a unique almost periodic solution $u : {\mathbb Z}\rightarrow C_{0}(\Omega)$ of problem (P). Furthermore, the sequences $k\mapsto m(k,\cdot)u(k,\cdot) \in C_{0}(\Omega),$ $k\in {\mathbb Z}$ and $k\mapsto \Delta_{0} u(k,\cdot) -b(k)u(k,\cdot) \in C_{0}(\Omega),$ $k\in {\mathbb Z}$ are almost periodic.

Let us finally note that we can use here the operators $B(k):=c(k)\Delta_{0}+d(k),$ $k\in {\mathbb Z}$, where $c(\cdot)$ and $d(\cdot)$ are almost periodic scalar-valued sequences  with sufficiently small sup-norms (cf. also Remark \ref{zat}), as well as that in place of the space $C_{0}(\Omega)$ and the operator $\Delta_{0}$ we can consider an arbitrary locally convex space $Y$ and an arbitrary subgenerator of an exponentially equicontionuous $(g_{\gamma},C)$-regularized solution operator family $(R_{\gamma}(t))_{t\geq 0}$ such that an estimate of type \eqref{tajp} is satisfied.
\end{example}

\section{Abstract nonautonomous higher-order difference equations}\label{s2nn}

In this section, we discuss the existence and uniqueness of almost periodic type solutions for certain classes of the abstract nonautonomous higher-order difference equations. First of all, we will briefly explain how Theorem \ref{zajeb-mlo1} and Theorem \ref{zajeb-mlo2} can be successfully applied in the analysis of the existence and uniqueness of systems of the abstract nonautonomous first-order difference equations (for further information concerning the operator matrices and their applications, we refer the reader to the book manuscript \cite{engel} by K. J. Engel). Let $p\in {\mathbb N}\setminus \{1\},$ and let us consider the following equations
\begin{align*}
{\mathrm C}{\mathrm B}(k+1)\vec{u}(k+1)={\mathrm A}(k)\vec{u}(k)+{\mathrm C}\vec{f}(k),\quad k\in {\mathbb Z}
\end{align*}
and
\begin{align}\label{karin1}
{\mathrm B}(k+1){\mathrm C}\vec{u}(k+1)={\mathrm A}(k)\vec{u}(k)+{\mathrm C}\vec{g}(k),\quad k\in {\mathbb Z},
\end{align}
where $\vec{f}(k) : {\mathbb Z} \rightarrow Y^{p}$ and $\vec{g}(k) : {\mathbb Z} \rightarrow Y^{p}$ are given sequences, $\vec{u}(k) : {\mathbb Z} \rightarrow Y^{p}$ is an unknown sequence,
$$
{\mathrm A}(k):=\left[ \begin{array}{ccccc}
A_{11}(k) & A_{12}(k)    & \cdot \cdot \cdot & A_{1p}(k) \\
A_{21}(k) & A_{22}(k)    & \cdot \cdot \cdot & A_{2p}(k)  \\
\cdot & \cdot  & \cdot \cdot \cdot & \cdot \\
A_{p-1;1}(k) & A_{p-1;2}(k)    & \cdot \cdot \cdot & A_{p-1;p}(k)  \\
A_{p1}(k) & A_{p2}(k)    & \cdot \cdot \cdot & A_{pp}(k)  \\  \end{array} \right],\quad k\in {\mathbb Z},
$$
$$
{\mathrm B}(k):=\left[ \begin{array}{ccccc}
B_{11}(k) & B_{12}(k)    & \cdot \cdot \cdot & B_{1p}(k) \\
B_{21}(k) & B_{22}(k)   & \cdot \cdot \cdot & B_{2p}(k)  \\
\cdot & \cdot  & \cdot \cdot \cdot & \cdot \\
B_{p-1;1}(k) & B_{p-1;2}(k)   & \cdot \cdot \cdot & B_{p-1;p}(k)  \\
B{p1}(k) & B_{p2}(k)    & \cdot \cdot \cdot & B_{pp}(k)  \\  \end{array} \right],\quad k\in {\mathbb Z},
$$
and 
$$
{\mathrm C}:=\left[ \begin{array}{ccccc}
C_{11} & C_{12}    & \cdot \cdot \cdot & C_{1p} \\
C_{21} & C_{22}    & \cdot \cdot \cdot & C_{2p}  \\
\cdot & \cdot  & \cdot \cdot \cdot & \cdot \\
C_{p-1;1} & C_{p-1;2}    & \cdot \cdot \cdot & C_{p-1;p} \\
C_{p1} & C_{p2}  & \cdot \cdot \cdot & C_{pp} \\  \end{array} \right],
$$
where $A_{ij}(k)$ and $B_{ij}(k)$ are linear operators on $Y$ and $C_{ij}\in L(Y)$ for $1\leq i,j\leq p$ and $k\in {\mathbb Z}.$ Then it is clear that Theorem \ref{zajeb-mlo1}, resp. Theorem \ref{zajeb-mlo2}, can be applied if the operator matrix $[{\mathrm A}(k)]^{-1}{\mathrm B}(k+1){\mathrm C},$ resp. ${\mathrm B}(k)[{\mathrm A}(k)]^{-1}{\mathrm C}$, belongs to $L(Y^{p})$ and all other requirements from the formulations of the above-mentioned results hold. We will illustrate the above consideration with the following concrete example: 

\begin{example}\label{karin}
Suppose that $(Y,\| \cdot \|)$ is a complex Banach space, ${\mathrm C}$ is the identity matrix on $Y^{p}$, $\vec{f}(k)=[f_{1}(k)\ ...\ f_{p}(k)]^{T}$, $k\in {\mathbb Z}$, $\vec{u}(k)=[u_{1}(k)\ ...\ u_{p}(k)]^{T}$, $k\in {\mathbb Z}$ and 
$$
g_{i}(k)=B_{i1}(k+1)f_{1}(k)+...+B_{ip}(k+1)f_{p}(k),\quad k\in {\mathbb Z},\ 1\leq i\leq p.
$$
Then \eqref{karin1} is equivalent with the following system of abstract difference equations of first order:
\begin{align}\notag
B_{i1}(k+1) & u_{1}(k+1)+...+B_{ip}(k+1)u_{p}(k+1)
\\\label{karin2}& =A_{i1}(k)u(k)+...+A_{ip}(k)u_{p}(k)+g_{i}(k),\quad k\in {\mathbb Z},\ 1\leq i\leq p.
\end{align}
Suppose that the following conditions hold:
\begin{itemize}
\item[(i)] The sequence $\vec{f}: {\mathbb Z} \rightarrow Y^{p}$ is almost periodic;
\item[(ii)] There exist operators $D_{ij}(k)\in L(Y)$ ($1\leq i,\ j\leq p;$ $k\in {\mathbb Z}$) such that the sequence $k\mapsto D_{ij}(k)\in L(Y)$, $k\in {\mathbb Z}$ is almost periodic ($1\leq i,\ j\leq p$), $\sum_{1\leq i,j\leq p}\| D_{ij}(k)\| \leq 1/(2p^{2})$ for all $k\in {\mathbb Z}$ and 
\begin{align}\label{bm}
B_{ij}(k+1)=A_{i1}(k)D_{1j}(k)+...+A_{ip}(k)D_{pj}(k),\quad k\in {\mathbb Z} \ \  (1\leq i,\ j\leq p).
\end{align} 
\end{itemize}
Then \eqref{bm} implies ${\mathrm B}(k+1)={\mathrm A}(k){\mathrm D}(k),$ $k\in {\mathbb Z},$ where ${\mathrm D}(k):=[D_{ij}(k)]_{1\leq i,j\leq p},$ and there exists an almost periodic solution of the system \eqref{karin2}; the uniqueness of almost periodic solutions holds if we additionally assume that the operators $[{\mathrm A}(k)]^{-1}$ are single-valued for all $k\in {\mathbb Z},$ when we have ${\mathrm D}(k)=[{\mathrm A}(k)]^{-1}{\mathrm B}(k+1),$ $k\in {\mathbb Z}$.
Finally, we would like to emphasize that our method is really practical since for any choice of almost periodic matrices of bounded linear operators ${\mathrm D}(\cdot)$
and for any choice of matrices of linear operators ${\mathrm A}(\cdot)$ we can directly compute the matrices ${\mathrm B}(\cdot)$ and clarify, after that, the type of system \eqref{karin2} which is solved. 

For example, if the set $\emptyset \neq \Omega \subseteq {\mathbb R}^{n}$ is open, bounded and Dirichlet regular, then the Dirichlet Laplacian $\Delta_{0}$
 generates a positive contractive strongly continuous semigroup on the space $Y=C_{0}(\Omega)$ so that there exists a finite real constant $M>0$ such that $\| (\lambda-\Delta_{0})^{-1} \| \leq M/(\Re \lambda),$ $\Re \lambda>0.$ Suppose now that for each $1\leq i,\ j\leq p$ the sequence $(b_{ij}(k))_{k\in {\mathbb Z}}$ is almost periodic and  
$$
\sum_{1\leq i,j\leq p}\Re (b_{ij}(k))\geq 2p^{2}M,\quad k\in {\mathbb Z}.
$$ 
Then Theorem \ref{zajeb-mlo2} is applicable with the operators $D_{ij}(k)=(b_{ij}(k)-\Delta_{0})^{-1}$,  $A_{ij}(k)$ being certain polynomials of the operator $\Delta_{0}$ and $B_{ij}(k)$ are given through \eqref{bm}, for $1\leq i,\ j\leq p$ and $k\in {\mathbb Z}$.
\end{example}

Let us consider now the following abstract nonautonomous higher-order difference equation:
\begin{align}
\notag CA_{p}(k&+p) u(k+p)+CA_{p-1}(k+p-1)u(k+p-1)+...
\\& \label{hi}+CA_{1}(k+1)u(k+1)+A_{0}u(k)=Cf(k),\quad k\in {\mathbb Z},
\end{align}
where $p\in {\mathbb N}\setminus \{1\},$ $f: {\mathbb Z} \rightarrow Y$, $C\in L(Y)$ and $A_{p}(k),...,A_{0}(k)$ are linear operators on $Y$ for all $k\in {\mathbb Z}.$ Set $\vec{u}(k):=[u(k)\ u(k+1)\, ... \, u(k+p-1)\bigr]^{T}$ and $\vec{f}(k):=[f(k)\ 0\, ... \, 0]^{T}$, $k\in {\mathbb Z}.$ Then we can rewrite \eqref{hi} in the following equivalent form:
\begin{align}\label{qaz}
{\mathbf C}{\mathbf B}(k+1)\vec{u}(k+1)={\mathbf A}(k)\vec{u}(k)+{\mathbf C}\vec{f}(k),\quad k\in {\mathbb Z},
\end{align}
where 
$$
{\mathbf A}(k):=\left[ \begin{array}{ccccc}
-A_{0}(k) & 0 & 0  & \cdot \cdot \cdot & 0 \\
0 & C  &  & \cdot \cdot \cdot & 0  \\
\cdot & \cdot &  \cdot & \cdot \cdot \cdot & \cdot \\
0 & 0 & 0 &  \cdot \cdot C &   0 \\
0 & 0 & 0 & \cdot \cdot \cdot & C  \end{array} \right],\quad k\in {\mathbb Z},
$$
$$
{\mathbf B}(k+1):=
\left[ \begin{array}{ccccc}
A_{1}(k+1) & A_{2}(k+2) & A_{3}(k+3)  & \cdot \cdot \cdot & A_{p}(k+p) \\
{\rm I} & 0 &  0 & \cdot \cdot \cdot & 0  \\
0 & {\rm I} &  \cdot & \cdot \cdot \cdot & 0 \\
0 & 0 & \cdot \cdot \cdot &  {\rm I} \ \ 0  & 0 \\
0 & 0 & 0 & \cdot \cdot {\rm I} & 0  \end{array} \right], \quad k\in {\mathbb Z}
$$
and ${\mathbf C}:=CI_{Y^{p}},$ where $I_{Y^{p}}$ denotes the identity operator on $Y^{p}.$ Direct calculations show that ($A_{2,1}=[A_{0}(k)]^{-1}C$):
\begin{align*}
&{\mathbf B}(k+1)\bigl[{\mathbf A}(k)\bigr]^{-1}{\mathbf C}
\\& =-\left[ \begin{array}{ccccc}
A_{1}(k)[A_{0}(k)]^{-1}C & A_{2}(k+1) & A_{3}(k+2)  & \cdot \cdot \cdot & A_{p}(k+p-1) \\
A_{2,1} & 0 &  0 & \cdot \cdot \cdot & 0  \\
0 & {-\rm I} &  \cdot & \cdot \cdot \cdot & 0 \\
0 & 0 & \cdot \cdot \cdot & -{\rm I} \ \ 0  & 0 \\
0 & 0 & 0 & \cdot \cdot -{\rm I} & 0  \end{array} \right], 
\end{align*}
for any $k\in {\mathbb Z}$, so that Theorem \ref{zajeb-mlo1} cannot be applied if $p\geq 3$ because of violation of condition \eqref{rac}; a similar analysis shows that is very difficult to apply Theorem \ref{zajeb-mlo2} for any $p\geq 2$. 

In the case that $p=2,$ we can apply Theorem \ref{zajeb-mlo1} and, as an outcome, we obtain the following result:

\begin{thm}\label{zajeb-mlo1ho}
Suppose that $p=2,$ $C\in L(Y)$ is injective, $f: {\mathbb Z} \rightarrow Y$ and for each integers $k\in {\mathbb Z}$ and $j\in \{0,1,2\}$ we have that $A_{j}(k)$ is a linear operator on $Y$ such that $[A_{0}(k)]^{-1}C\in L(Y),$ $A_{1}(k)[A_{0}(k)]^{-1}C\in L(Y)$ and $A_{2}(k)\in L(Y).$
Suppose, further, that
for each $k\in {\mathbb Z}$ and $\kappa \in \circledast_{Y}$ there exist real numbers $c^{\kappa}(k,1)>0$, $c^{\kappa}(k,2)>0$ and $c^{\kappa}(k,3)>0$
such that $\kappa([A_{0}(k)]^{-1}Cx)\leq c^{\kappa}(k,1)\kappa(x)$, $\kappa(A_{1}(k)[A_{0}(k)]^{-1}Cx )\leq c^{\kappa}(k,2)\kappa(x)$ and $\kappa(A_{2}(k)x)\leq c^{\kappa}(k,3)\kappa(x)$ for all $x\in Y,$ $k\in {\mathbb Z}$ and $\kappa \in \circledast_{Y}.$
Set $c^{\kappa}(k):=c^{\kappa}(k,1)+c^{\kappa}(k,2)+c^{\kappa}(k,3)$ for all $\kappa \in \circledast_{Y}$ and $k\in {\mathbb Z}.$
Assume that \eqref{rac} holds
and for each $\kappa \in \circledast_{Y}$ we have $\sup_{k\in {\mathbb Z}}c^{\kappa}(k)<+\infty .$ 
Then we have the following:
\begin{itemize}
\item[(i)] If the sequence $f : {\mathbb Z} \rightarrow Y$ is bounded, then the sequence $\vec{v}(\cdot),$ given by
\begin{align}
\notag \vec{v}(k):&=-{\mathbf B}(k)\bigl[{\mathbf A}(k)\bigr]^{-1}{\mathbf C}\vec{f}(k)
\\\label{milevaimho} &-\sum_{j=-\infty}^{k-2}\Biggl[ \Bigl[ {\mathbf B}(k)\bigl[{\mathbf A}(k)\bigr]^{-1}{\mathbf C}\Bigr]\cdot ...\cdot \Bigl[ {\mathbf B}(-j-1)\bigl[{\mathbf A}(-j-1)\bigr]^{-1}{\mathbf C}\Bigr]\Biggr] \vec{f}(-j-1),
\end{align}
for any $k\in {\mathbb Z},$
is well-defined and $\vec{v}(\cdot)$ is a solution of
\begin{align}\label{vbmfg}
{\mathbf C}\vec{v}(k+1)\in \Bigl[{\mathbf A}(k)\bigl[ {\mathbf B}(k)\bigr]^{-1}\Bigr]\vec{v}(k)+{\mathbf C}\vec{f}(k),\quad k\in {\mathbb Z}.
\end{align}
Furthermore, the sequence ${\mathbf B}(\cdot)\vec{v}(\cdot)$ is a solution of \eqref{qaz} and its first projection is a solution of \eqref{hi}.
\item[(ii)] If the sequence $k\mapsto {\mathbf B}(k)[{\mathbf A}(k)]^{-1}{\mathbf C}\vec{f}(k),$ $k\in {\mathbb Z}$ is Bohr almost periodic and, for every relatively compact set $B\subseteq Y^{k},$ $\kappa \in \circledast_{Y^{k}}$ and $\epsilon>0,$  there exists a finite real number
$L>0$ such that for each $t_{0}\in {\mathbb Z}$ there exists $\tau \in B(t_{0},L)\cap {\mathbb Z}$ such that
$$
\kappa \Bigl( { \mathbf B}(k+\tau)\bigl[ { \mathbf A}(k+\tau)\bigr]^{-1} { \mathbf C}\vec{x}- { \mathbf B}(k)\bigl[ { \mathbf A}(k)\bigr]^{-1}C\vec{x}\Bigr)\leq \epsilon,\quad k\in {\mathbb Z},\ \vec{x}\in B,
$$
then a unique almost periodic solution $\vec{v}(\cdot)$ of \eqref{vbmfg} is given by \eqref{milevaimho}.  Furthermore, the sequences $ { \mathbf A}(\cdot)\vec{v}(\cdot)$ and $ { \mathbf B}(\cdot)\vec{v}(\cdot)$ are Bohr almost periodic and there exists a unique Bohr almost periodic solution of \eqref{hi}.
\item[(iii)] If $\omega \in {\mathbb N},$ $c\in {\mathbb C}\setminus \{0\},$ the sequence $k\mapsto f(k),$ $k\in {\mathbb Z}$ is $(\omega,c)$-periodic, $ A(j)(k+\omega)= A(j)(k)$ and $ B_{j}(k+\omega)=  B_{j}(k)$ for all $k\in {\mathbb Z}$ and $j=0,1,2,$ 
then an $(\omega,c)$-periodic solution $\vec{v}(\cdot)$ of \eqref{vbmfg} is given by \eqref{milevaimho} and an $(\omega,c)$-periodic solution of \eqref{qaz}, resp., \eqref{hi}, is given by $ { \mathbf B}(\cdot)\vec{v}(\cdot)$, resp. by its first projection. The uniqueness of  $(\omega,c)$-periodic solutions to \eqref{qaz} and \eqref{hi} holds if $|c|\geq 1.$
\end{itemize}
\end{thm}

Without going into full details, we will only note that Theorem \ref{zajeb-mlo1ho} can be successfully applied in the analysis of the existence and uniqueness of almost periodic solutions to 
the following abstract nonautonomous semi-discrete Poisson wave equation in the space $C_{0}(\Omega)$: 
\begin{align*} m_{2}(k+2,x)& u(k+2,x)+
m_{1}(k+1,x)u(k+1,x)
\\& =\Delta_{0} u(k,x) -b(k)u(k,x)+f(k,x),\ k\in {\mathbb Z},\ x\in {\Omega};\\&
u(k,x)=0,\quad (k,x)\in {\mathbb Z} \times \partial \Omega ,
\end{align*}
where the sequences $m_{1,2}: {\mathbb Z}\rightarrow C_{0}(\Omega)$ are almost periodic, the sequence $f : {\mathbb Z}\rightarrow C_{0}(\Omega)$ is almost periodic, the sequence $b : {\mathbb Z}\rightarrow {\mathbb C}$ is almost periodic, the values of terms $\sup_{k\in {\mathbb Z};x\in \overline{\Omega}}|m_{1,2}(k,x)|$ are sufficiently small and there exists $c>0$ such that $\Re (b(k))>c$ for all $k\in {\mathbb Z}.$  The consideration is quite similar to the consideration carried out in Example \ref{dl} and therefore omitted.

\end{document}